\documentclass[12pt,a4paper,reqno,oneside]{amsart}
\usepackage{t1enc}
\usepackage{times}
\usepackage{amssymb}
\usepackage[mathscr]{euscript}
\usepackage{graphicx}
\usepackage{hyperref}

\usepackage{float}
\usepackage{epstopdf}
\usepackage{pdfpages}
\newtheorem{thm}{Theorem}[section]
\newtheorem{prop}[thm]{Proposition}
\newtheorem{lem}[thm]{Lemma}
\newtheorem{cor}[thm]{Corollary}

\theoremstyle{remark}
\newtheorem{rem}[thm]{Remark}
 \newtheorem{ex}[thm]{Example}

\theoremstyle{definition}

\newcommand*{\rom}[1]{\expandafter\@slowromancap\romannumeral #1@} 
\renewcommand{\phi}{\varphi} 
\newcommand{\tr}{\mathrm{Tr}} 
\newcommand{\R}{\mathbb R} 

\newcommand{\nooutput}[1]{}

\begin{document}

\title[The Bismut-Elworthy-Li formula for mean-field SDEs]{The Bismut-Elworthy-Li formula for mean-field stochastic differential equations}
\date{\today}

\author[Banos]{David Ba\~nos}
\address[David Ba\~nos]{\\
Department of Mathematics \\
University of Oslo\\
P.O. Box 1053, Blindern\\
N--0316 Oslo, Norway}
\email[]{davidru@math.uio.no}


\begin{abstract}
 We generalise the so-called Bismut-Elworthy-Li formula to a class of stochastic differential equations whose coefficients might depend on the law of the solution. We give some examples of where this formula can be applied to in the context of finance and the computation of Greeks and provide with a simple but rather illustrative simulation experiment showing that the use of the Bismut-Elworthy-Li formula, also known as Malliavin method, is more efficient compared to the finite difference method.
\end{abstract}

\maketitle

\section{Introduction}
\label{Intro}
It is known that the spatial derivative of the solution to the (backward) Kolmogorov equation can be represented as an expectation of a functional of the solution of an SDE with some weight, namely the so-called Bismut-Elworthy-Li (BEL) formula as shown in \cite{Bismut.84} and extended in \cite{El-Li.94}. In \cite{FLLLT} the authors use techniques from Malliavin calculus to prove BEL formula and employ it for the computation of sensitivities of financial options, also known as \emph{Greeks}.

In many applications, it is very natural to expect that the coefficients of a stochastic differential equation (SDE) may depend on properties of the law of the solution, such as dependence on its moments. Here, we want to extend the formula to mean-field type SDEs following the essence of \cite{FLLLT} and show that such generalisation is actually non-trivial, requiring more regularity of the solution in the sense of Malliavin. First, we give a relationship between the Malliavin derivative and the spatial derivative of the solution with respect to the initial condition. Already here we see that such generalisation involves an extra factor which is no longer adapted, thus requiring more (Malliavin) regularity on the solution which is not immediate. Fortunately, if $b$ and $\sigma$ are Lipschitz continuous in space, then the solution is twice Malliavin differentiable, as it is shown in \cite{banos.nilssen.14}, and hence a formula using the Skorokhod integral may be expected. Using such relation one can find the BEL formula in this context. Some merely illustrative examples are provided in order to give a better insight on the effect of mean-field SDEs in the BEL formula. In the last examples we carry out some simulations to show that the Malliavin method is more efficient compared to a finite difference method, especially when the function involved is discontinuous.

The paper is organised as follows: In Section \ref{Frame} we collect some summarised basic facts on Malliavin Calculus needed for the derivation of the main results of the paper. In Section \ref{SectionBEL} we include all intermediate steps towards the main result which is the Bismut-Elworthy-Li formula in the context of mean-field SDEs. Finally, Section \ref{SectionAppl} is devoted to provide some illustrative examples of this generalised Bismut-Elworthy-Li formula with simulations. The findings are similar to those in \cite{FLLLT}, the use of Bismut-Elworthy-Li's formula serves as a much more efficient  method to compute sensitivities with respect to initial data, especially when the function involved in the expectation is highly irregular.

{\bf Notations:} Let $\R_+$ denote the non-negative real numbers. Denote by $|\cdot|$ the Euclidean norm in $\R^d$, $d\geq 1$. Given a Banach space $E$, denote by $\|\cdot\|_E$ its associated norm. Let $k,p\geq 0$ integers and $\mathbb{D}^{k,p}$ be the space of $k$ times Malliavin differentiable random variables with all $p$-moments. Denote by $D_s$, $s\geq 0$ the Malliavin derivative as introduced in \cite[Chapter 1, Section 1.2.1]{Nua10} and $\delta$ its dual operator (Skorokhod integral). Denote by $\mbox{Dom } \delta$ the domain of $\delta$ (Skorokhod integrable processes). Denote the trace of a matrix $M\in\mathbb R^{d\times d}$ by $\tr(M):=\sum_{j=1}^d M_{j,j}$ and by $M^\ast$ its transpose. For a (weakly) differentiable function $f:\R^d\times \R^d \rightarrow \R^d$, $(x,y)\mapsto f(x,y)$, denote by $\partial_1$, respectively by $\partial_2$, (weak) differentiation with respect to the first (space) variable $x \in \R^d$, respectively the second (space) variable $y\in \R^d$.

\section{Framework}\label{Frame}

Our main results centrally rely on tools from Malliavin calculus. We here provide a concise introduction to the main concepts in this area. For deeper information on Malliavin calculus the reader is referred to i.e.~\cite{DOP08, Mall78, Mall97, Nua10}.

\subsection{Malliavin calculus} \label{App1}  
Let $W=\{W_t, t\in [0,T]\}$ be a standard Wiener process on some complete filtered probability space $(\Omega, \mathfrak{A}, P)$ where $\mathcal{F}=\{\mathcal{F}_t\}_{t\in [0,T]}$ is the $P$-augmented natural filtration. Denote by $\mathcal{S}$ the set of simple random variables $F\in L^2(\Omega)$ in the form
$$F=f\left( \int_0^T h_1(s)dW_s, \dots, \int_0^T h_n(s)dW_s\right), \quad h_1,\dots, h_n\in L^2([0,T]), \ \ f\in C_0^{\infty} ( \R^n).$$
The Malliavin derivative operator $D$ acting on such simple random variables is the process $DF = \{D_t F, t\in [0,T]\}$ in $L^2(\Omega\times [0,T])$ defined by
\begin{equation*}
D_t F = \sum_{i=1}^{n}\partial
_{i}f\left( \int_0^T h_1(s)dW_s, \dots, \int_0^T h_n(s)dW_s\right)h_{i}(t).
\end{equation*}%
Define the following norm on $\mathcal{S}$:
\begin{align}\label{Mallnorm}
\left\Vert F\right\Vert _{1,2}:=\left\Vert F\right\Vert _{L^{2}(\Omega
)}+\left\Vert DF\right\Vert _{L^{2}(\Omega ;L^2([0,T]))}=E[|F|^2]^{1/2}+ E\left[ \int_0^T |D_t F|^2 dt\right]^{1/2}.
\end{align}%

We denote by $\mathbb{D}^{1,2}$ the closure of the family of simple random variables $\mathcal{S}$ with respect to the norm given in \eqref{Mallnorm} and we will refer to this space as the space of Malliavin differentiable random variables in $L^2(\Omega)$ with Malliavin derivative belonging to $L^2(\Omega)$.

\ \

In the derivation of the probabilistic representation for the Delta, the following chain rule for the Malliavin derivative will be essential:
\begin{lem}\label{chainrule}
Let $\varphi:\R^m\rightarrow \R$ be continuously differentiable with bounded partial derivatives. Further, suppose that $F=(F_1,\dots, F_m)$ is a random vector whose components are in $\mathbb{D}^{1,2}$. Then $\varphi(F)\in \mathbb{D}^{1,2}$ and 
$$D_t \varphi(F) = \sum_{i=1}^m \partial_i \varphi(F) D_t F_i, \quad P-a.s., \quad t\in [0,T].$$
\end{lem}

The Malliavin derivative operator $D: \mathbb{D}^{1,2} \rightarrow L^2(\Omega\times[0,T])$ admits an adjoint operator $\delta = D^\ast: \mbox{Dom} (\delta) \rightarrow L^2(\Omega)$ where the domain $\mbox{Dom}(\delta)$ is characterised by all $u\in L^2(\Omega\times[0,T])$ such that for all $F\in \mathbb{D}^{1,2}$ we have
$$E\left[ \int_0^T D_t F \, u(t) dt \right] \leq C \|F\|_{1,2},$$
where $C$ is some constant depending on $u$.

For a stochastic process $u \in \mbox{Dom}( \delta)$ (not necessarily adapted to $\mathcal{F}$) we denote by
\begin{align}\label{skorokhod}
\delta(u) := \int_0^T u(t) \delta W_t
\end{align}
the action of $\delta$ on $u$. The above expression (\ref{skorokhod}) is known as the Skorokhod integral of $u$ and it is an anticipative stochastic integral. It turns out that all $\mathcal{F}$-adapted processes in $L^2(\Omega\times[0,T])$ are in the domain of $\delta$ and for such processes $u$ we have
$$\delta( u ) =\int_0^T u(t) dW_t,$$
i.e. the Skorokhod and It\^{o} integrals coincide. In this sense, the Skorokhod integral can be considered to be an extension of the It\^{o} integral to non-adapted integrands.

As for the It\^{o} integral, there is also a corresponding isometry property for the Skorokhod integral. The proof of this can e.g. be found in \cite[Theorem 6.17.]{DOP08}.
\begin{thm}
Let $u$ be a process such that $u(s)\in \mathbb{D}^{1,2}$ for a.e. $s\in [0,T]$ and
$$E\left[ \int_0^T u(t)^2 dt+ \int_0^T \int_0^T |D_t u(s) D_s u(t)| dsdr\right] <\infty.$$
Then $u$ is Skorokhod integrable and
$$E\left[ \left(\int_0^T u(t) \delta B_t\right)^2\right] = E\left[ \int_0^T u(t)^2 dt+ \int_0^T \int_0^T D_t u(s) D_s u(t) dsdr\right].$$
\end{thm}

The dual relation between the Malliavin derivative and the Skorokhod integral implies the following important formula:

\begin{thm}[Duality formula]\label{duality}
Let $F\in \mathbb{D}^{1,2}$ be $\mathcal{F}_T$-measurable and $u\in \mbox{Dom}(\delta)$. Then
\begin{align}
E\left[ F\int_0^T u(t) \delta W_t\right] = E\left[ \int_0^T u(t) D_t F dt\right].
\end{align}
\end{thm}

The following is the corresponding integration by parts formula for the Skorokhod integral. See e.g. \cite[Theorem 3.15.]{DOP08}.

\begin{thm}[Integration by parts]\label{IBP}
Let $u\in \mbox{Dom}(\delta)$ and $F\in \mathbb{D}^{1,2}$ such that $Fu \in \mbox{Dom}(\delta)$. Then
$$F\int_0^T u(s) \delta W(t) = \int_0^T Fu(t) \delta W(t) + \int_0^T u(t) D_t F dt.$$
\end{thm}

\section{The (mean-field) Bismut-Elworthy-Li formula}\label{SectionBEL}
The object of study is a \emph{mean-field} type \emph{stochastic differential equation} (SDE) of the form
\begin{align}\label{SDE}
\begin{split}
dX_t =& b(t,X_t,\rho_t) dt + \sum_{k=1}^m \sigma_k(t,X_t, \pi_t)dW_t^k, \quad X_0^x=x\in \R^d, \quad t\in [0,T]\\
\rho_t :=& \, E[\varphi(X_t)], \quad \pi_t := E[\psi(X_t)]
\end{split}
\end{align}
where $T\in \R$, $T>0$, $b: [0,T]\times \R^d \times \R^d \rightarrow \R^d$, $\sigma_k:[0,T]\times \R^d \times \R^d \rightarrow \R^d$, $k=1,\dots, m$, $\varphi:\R^d\rightarrow \R^d$, $\psi:\R^d\rightarrow \R^d$ are measurable functions and $W=\{W_t, t\in [0,T]\}$ is an $m$-dimensional Brownian motion on some probability space $(\Omega, \mathfrak{A}, P)$ equipped with the natural filtration augmented by all $P$-null sets, denoted by $\mathcal{F}=\{\mathcal{F}_t\}_{t\in [0,T]}$.

We will usually consider the solution as a function of $x$ and hence write $X_t^x$ to stress this fact. Otherwise, we will just write $X_t$. Moreover, we will assume the following conditions as in \cite{BLPR.14}
\begin{itemize}
\item[(i)] the functions $(t,x,y)\mapsto b(t,x,y)$ and $(t,x,y)\mapsto \sigma_k(t,x,y)$, $k=1,\dots,m$ are continuously differentiable with bounded Lipschitz derivatives uniformly with respect to $t\in [0,T]$.

\item[(ii)] Assume $d\leq m$ and the matrix $(\sigma_1 , \dots, \sigma_m)$ is uniformly elliptic and admits a right pseudo-inverse.

\item[(iii)] The functions $\varphi$ and $\psi$ are continuously differentiable with bounded Lipschitz derivatives.
\end{itemize}

The following proposition is extracted from the work by Buckdahn, Li, Peng and Rainer in \cite{BLPR.14} where they show that under the hypothesis above, equation \eqref{SDE} admits a stochastic flow of diffeomorphisms, in particular, the function $x\mapsto X_t^x$ is $P$-a.s. classically differentiable. This property is crucial to define the first variation process and relate it to the Malliavin derivative of the solution.

\begin{prop}\label{diff}
Let $X=\{X_t, t\in [0,T]\}$ be the unique global strong solution of (\ref{SDE}). Then the function $x\mapsto X_t^x$ is continuously differentiable, $P$-a.s.
\end{prop}
\begin{proof}
See \cite{BLPR.14}.
\end{proof}

Next proposition shows that the first variation process is invertible for every $t>0$.

\begin{prop}\label{detYt}
Let $Y=\{Y_t, t\in [0,T]\}$ be the solution to the following matrix-valued linear SDE
$$dY_t = A_t Y_t dt + \sum_{k=1}^m B_t^k Y_t dW_t^k, \quad Y_0=I, \quad t\in [0,T]$$
where $A_t := \partial_1 b(t,X_t^x, \rho_t^x)$ and $B_t^k :=\partial_1 \sigma_k (t,X_t^x, \pi_t^x)$, $k=1,\dots,m$. Then
$$(\det Y_t)^{-1} \in \bigcap_{p\geq 1} L^p (\Omega).$$
As a consequence, $Y_t$ is $P$-a.s. invertible for every $t\in [0,T]$.
\end{prop}
\begin{proof}
We want to show that
$$E\left[ \left| \det Y_t\right|^{-p} \right] <\infty$$
for every integer $p\geq 1$. Indeed, in virtue of (stochastic) Liouville's formula which can be found in \cite{Ivo.78}, one has
$$\det Y_t = \exp \left\{ \int_0^t \left(\tr A_u + \frac{1}{2} \sum_{k=1}^m (\tr B_u^k )^2 \right) du + \sum_{k=1}^m \int_0^t \tr B_u^k dW_u^k \right\}, \quad P-a.s.$$

Observe that since the processes $\tr B^k$ are in $L^2([0,T])$ and $\mathcal{F}$-adapted, the stochastic integrals appearing in the exponent are martingales. Hence,
$$M_t := \exp \left\{\sum_{k=1}^m \int_0^t \tr B_u^k dW_u^k-\frac{1}{2}\sum_{k=1}^m \int_0^t (\tr B_u^k)^2 du\right\}$$
is an $\mathcal{F}$-martingale with $E[M_t] = E[M_0]=1$. Thus, using Cauchy-Schwarz' inequality we have for every $p\geq 1$
\begin{align*}
E[|\det Y_t|^{-p}] \leq E\left[\exp \left\{(2p^2-p)\sum_{k=1}^m \int_0^t \tr B_u^k dW_u^k-2p \int_0^t \tr A_u du\right\} \right].
\end{align*}

In particular, the claim is reduced to showing that
$$\sup_{t\in [0,T]} \left|E\left[\exp\left\{\lambda \int_0^t \tr A_u du\right\}\right]\right| + \sup_{t\in [0,T]} \left|E\left[\exp\left\{\lambda \int_0^t \sum_{k=1}^m (\tr B_u^k)^2 du\right\}\right]\right|<\infty$$
for every $\lambda \in \R$ which clearly holds since $A$ and $B^k$, $k=1,\dots,m$ are uniformly bounded.
\end{proof}

The following statement is one of the main observations for the derivation of the Bismut-Elworthy-Li formula in the mean-field context. It can be seen as a generalisation of the well-known relation between the first variation process $Y$ in the non mean-field context and the Malliavin derivative, see e.g. \cite[Ch.2, Sec.2.3.1]{Nua10}.

\begin{thm}\label{MallSob}
Let $X=\{X_t, t\in [0,T]\}$ be the solution of \eqref{SDE}. Then for every $s,t \in [0,T]$, $s\leq t$ one has the following relationship between the spatial derivative and the Malliavin derivative of $X_t^x$
\begin{align}\label{relation}
\frac{\partial}{\partial x} X_t^x = D_s X_t^x \sigma^{-1}(s,X_s^x,\pi_s^x)Y_s u(t) \quad s\leq t,
\end{align}
where
$$u(t):=\left( I + \int_0^t Y_u^{-1}\left(\alpha_u - \sum_{k=1}^m B_u^k \beta_u^k \right) du + \sum_{k=1}^m \int_0^t Y_u^{-1} \beta_u^k dW_u^k \right),$$
$\sigma^{-1}$ denotes the right pseudo-inverse of $\sigma$, $Y=\{Y_t, t\in [0,T]\}$ is the $d \times d$ fundamental matrix satisfying
$$dY_t = A_t Y_t dt + \sum_{k=1}^m B_t^k Y_t dW_t^k, \quad Y_0=I, \quad t\in [0,T]$$
and where $A=\{A_t, t\in [0,T]\}$, $\alpha= \{\alpha_t, t\in [0,T]\}$, $B^k=\{B_t^k, t\in [0,T]\}$, $\beta^k=\{\beta_t^k, t\in [0,T]\}$, $k=1,\dots, m$ are matrix valued processes defined as:
\begin{align*}
A_t :=& \, \partial_1 b(t,X_t^x, \rho_t^x), \quad \alpha_t := \partial_2 b(t,X_t^x, \rho_t^x)\frac{\partial}{\partial x}\rho_t^x \\
B_t^k :=& \, \partial_1 \sigma_k (t,X_t^x, \pi_t^x), \quad \beta_t^k := \partial_2 \sigma_k(t,X_t^x, \pi_t^x)\frac{\partial}{\partial x}\pi_t^x
\end{align*}
for $k=1,\dots,m$.
\end{thm}
\begin{proof}
Differentiating with respect to $x\in \R^d$ we have that $\frac{\partial}{\partial x} X_t^x$ satisfies the following matrix-valued linear equation
\begin{align}\label{matrixJacobian}
\begin{split}
\frac{\partial}{\partial x} X_t^x =& \, I + \int_0^t \left(\partial_1 b(u,X_u^x , \rho_u^x) \frac{\partial}{\partial x} X_u^x + \partial_2 b(u,X_u^x , \rho_u^x) \frac{\partial}{\partial x}\rho_u^x\right) du \\
&+ \sum_{k=1}^m \int_0^t \left(\partial_1 \sigma_k(u,X_u^x , \pi_u^x) \frac{\partial}{\partial x} X_u^x + \partial_2 \sigma_k(u,X_u^x , \pi_u^x) \frac{\partial}{\partial x}\pi_u^x\right) dW_u^k.
\end{split}
\end{align}
Using the notations in the statement of the theorem, we can solve (\ref{matrixJacobian}) and express $\frac{\partial}{\partial x} X_t$ as
$$\frac{\partial}{\partial x} X_t = Y_t \left( I + \int_0^t Y_u^{-1}\left(\alpha_u - \sum_{k=1}^m B_u^k \beta_u^k \right) du + \sum_{k=1}^m \int_0^t Y_u^{-1} \beta_u^k dB_u^k \right)$$
where $Y=\{Y_t, t\in [0,T]\}$ is the $d\times d$ fundamental matrix satisfying $Y_0 = I$ and
$$dY_t = A_t Y_t dt + \sum_{k=1}^m B_t^k Y_t dW_t^k, \quad t\in [0,T].$$

By the well-known classical relation, see e.g. \cite[Chapter 2, Section 2.3.1]{Nua10}, it is true that,
$$Y_t = D_s X_t \sigma^{-1}(s,X_s^x, \pi_s^x) Y_s, \quad s\leq t$$
where $\sigma^{-1}$ denotes the right pseudo-inverse of $\sigma$ and hence the relation follows.
\end{proof}

\begin{rem}
For the relation $Y_t=D_s X_t \sigma^{-1}(s,X_s,\pi_s)Y_s$, $s\leq t$ to hold in the mean-field setting one also needs the property that $x\mapsto X_t^x$ defines a stochastic semiflow. In the mean-field case we point out that the fact that $b$, $\sigma$, $\varphi$ and $\psi$ are continuously differentiable with bounded Lipschitz derivative implies this fact in virtue of \cite{BLPR.14}.
\end{rem}

It is shown in \cite{banos.nilssen.14} that SDE (\ref{SDE}) is twice Malliavin differentiable when the vector field $b$ does not depend on the law of $X$ and one has additive noise. Nevertheless, using the same method one can prove the same result since the dependence on $E[\varphi(X_t)]$ and $E[\psi (X_t)]$ does not bring stochasticity to the equation. In the sense that the Malliavin derivative of $X_t$ for every fixed $t\in [0,T]$ takes the same form as in the usual linear setting.

Henceforth, we will assume the following technical condition for simplicity.

\begin{itemize}
\item There exists a bijection $\Lambda_t :\R^d \rightarrow \R^d$, for every $t\in [0,T]$, such that the function $b_\ast:[0,T] \times \R^d\times \R^d \times \R^d \rightarrow \R^d$ defined as
\begin{align*}
b_\ast (t,x,\rho_t,\pi_t) :=& \, \partial_t \Lambda_t(\Lambda_t^{-1}(x)) +\partial_x \Lambda_t (\Lambda_t^{-1}(x))[b(t,\Lambda_t^{-1}(x),\rho_t)]\\
&+ \frac{1}{2} \partial_x^2 \Lambda_t(\Lambda_t^{-1}(x)) \left[ \sum_{i=1}^d \sigma(t,\Lambda_t^{-1}(x),\pi_t)[e_i] , \sum_{i=1}^d \sigma(t,\Lambda_t^{-1}(x),\pi_t)[e_i]\right]
\end{align*}
where $\{e_i\}_{i=1,\dots, d}$ is a basis of $\R^d$, is Lipschitz continuous uniformly with respect to $t\in [0,T]$.
\end{itemize}
The reason of the above condition is to use It\^{o}'s formula on the process $Z_t=\Lambda(X_t)$ so that $Z$ satisfies an SDE with additive noise for which the results from \cite{banos.nilssen.14} can be applied. Although, it might seem that the class of such processes is small, it covers a wide variety of models which are relevant in applications, such as for instance geometric-type models.

\begin{prop}\label{Malldiff}
Let $X=\{X_t, t\in [0,T]\}$ be the unique global strong solution of (\ref{SDE}). Then we have
$$X_t \in \bigcap_{p\geq 1} \mathbb{D}^{2,p}(\Omega)$$
for every $t\in [0,T]$.
\end{prop}
\begin{proof}
See \cite{banos.nilssen.14}.
\end{proof}

\begin{prop}\label{SkorokhodI}
Let $t\in [0,T]$ and define
$$F:=\int_0^t Y_u^{-1}\left(\alpha_u - \sum_{k=1}^m B_u^k \beta_u^k \right) du, \quad G:= \sum_{k=1}^m \int_0^t Y_u^{-1} \beta_u^k dB_u^k,$$
then $F,G \in \mbox{Dom }\delta$.
\end{prop}
\begin{proof}
Since $\mathbb{D}^{1,2}\subset \mbox{Dom }\delta$, see \cite[Proposition 1.3.1.]{Nua10}, one needs to show that $F,G\in \mathbb{D}^{1,2}$. In both cases it suffices to show that the integrands are Malliavin differentiable. For the Lebesgue integral it is immediate as for the It\^{o} integral one may justify this by \cite[Lemma 1.3.4.]{Nua10}.

Now, since $Y_t = D_s X_t Y_s$, $s\leq t$ and $X_t \in \mathbb{D}^{2,2}$ for every $t\in [0,T]$ we have together with Proposition \ref{detYt} that $Y_t^{-1} \in \mathbb{D}^{1,2}$ for every $t\in [0,T]$. The result follows since the functions $\alpha$, $B^k$ and $\beta^k$ are continuously differentiable in the first variable with bounded derivative.
\end{proof}

\begin{cor}\label{SkorokhodII}
The process
$$s\mapsto \sigma^{-1}(s,X_s^x,\pi_s^x)Y_s \left( I + \int_0^t Y_u^{-1}\left(\alpha_u - \sum_{k=1}^m B_u^k \beta_u^k \right) du + \sum_{k=1}^m \int_0^t Y_u^{-1} \beta_u^k dW_u^k \right),$$
$0\leq s \leq t$, is Skorokhod integrable.
\end{cor}
\begin{proof}
Indeed, it is the product of an adapted process, hence Skorokhod integrable and by Proposition \ref{SkorokhodI} a Skorokhod integrable random variable.
\end{proof}

\begin{thm}[Bismut-Elworthy-Li formula]\label{BELthm}
Let $X=\{X_t, t\in [0,T]\}$ be the unique global strong solution of (\ref{SDE}). Let $\Phi: \R^d \rightarrow \R_+$ be a measurable function such that $\Phi(X_t^x)\in L^2(\Omega)$. Define the function
$$v(x) := E[\Phi(X_t^x)].$$
Then
\begin{align}\label{BEL}
\frac{\partial}{\partial x} v(x) =E\left[\Phi(X_t^x) \int_0^t a(s) \left[\sigma^{-1}(s,X_s^x,\pi_s^x)Y_s u(t)\right]^\ast\delta W_s\right]^\ast
\end{align}
where $\ast$ denotes transposition, $Y$ is the fundamental matrix obtained in Theorem \ref{MallSob} and
$$u(t):= I + \int_0^t Y_u^{-1}\left(\alpha_u - \sum_{k=1}^m B_u^k \beta_u^k \right) du + \sum_{k=1}^m \int_0^t Y_u^{-1} \beta_u^k dW_u^k,$$
here $a:[0,T] \rightarrow \R$ is an integrable function such that $\int_0^t a(s)ds = 1$ and $a\equiv 0$ otherwise. The functions $\alpha$, $B^k$, $\beta^k$, $k=1,\dots,m$ are defined as in Theorem \ref{MallSob}.

In fact, the Skorokhod integral in \eqref{BEL} can be expressed as an It\^{o} integral plus a finite variation term. Indeed, since $u(t) \in \mathbb{D}^{1,2}$, integration by parts for the Skorokhod integral, see Theorem \ref{IBP} or \cite[Theorem 3.15]{DOP08}, yields
\begin{align}\label{BEL2}
\begin{split}
\frac{\partial}{\partial x} v(x) =& \, E\Bigg[\Phi(X_t^x) \Bigg(\int_0^t a(s) \left[\sigma^{-1}(s,X_s^x,\pi_s^x)Y_s\right]^\ast dW_s \, u(t)\\
&\hspace{4cm}- \int_0^t a(s) [\sigma^{-1}(s,X_s^x,\pi_s^x)Y_s D_s u(t)]^{\ast} ds\Bigg)\Bigg]^\ast.
\end{split}
\end{align}
\end{thm}
\begin{proof}
We will carry out the proof in four steps. First, we will show the formula for smooth functions $\Phi$ with compact support. Then extend it to any continuous and bounded function $\Phi$ by using a limit argument. We then get rid of the continuity by employing a monotone class argument. Finally, we consider any general function with the property that $\Phi (X_t^x)\in L^2(\Omega)$.

{\bf Step 1:} Assume first that $\Phi$ is infinitely differentiable with compact support. By Theorem \ref{MallSob} we have
$$\frac{\partial}{\partial x} X_t = D_s X_t \sigma^{-1}(s,X_s^x,\pi_s^x)Y_s u(t) , \quad s\leq t , \quad P-a.s.$$

Then multiplying both sides by the function $a$ and integrating over $s\in [0,t]$ we have
\begin{align}\label{Mallintrel}
\frac{\partial}{\partial x} X_t^{x} = \int_0^t a(s) D_s X_t \sigma^{-1}(s,X_s^x,\pi_s^x)Y_s u(t) ds, \quad P-a.s.
\end{align}
As a consequence,
\begin{align*}
\frac{\partial}{\partial x} v(x) =& \, E\left[ \Phi ' (X_t^{x}) \frac{\partial }{\partial x} X_t^{x}\right] \\
=& \, E\left[ \Phi ' (X_t^{x}) \int_0^t a(s) D_s X_t \sigma^{-1}(s,X_s^x,\pi_s^x)Y_s u(t) ds\right] \\
=& \, E\left[  \int_0^t a(s) D_s \Phi(X_t) \sigma^{-1}(s,X_s^x,\pi_s^x)Y_s u(t) ds\right]\\
=& \, E\left[ \Phi(X_t) \int_0^t a(s) \left[\sigma^{-1}(s,X_s^x,\pi_s^x)Y_s u(t)\right]^\ast \delta B_s\right]^\ast
\end{align*}
where we have used relation (\ref{Mallintrel}), the chain rule for the Malliavin derivative (backwards) and the duality formula for the Malliavin derivative which is justified by Corollary \ref{SkorokhodII}.

{\bf Step 2:} Assume $\Phi$ is bounded and continuous, in particular, $\Phi(X_t^x)\in L^2(\Omega)$. We can approximate $\Phi$ by a sequence of smooth functions $\{\Phi_n\}_{n\geq 0}$ with compact support such that $\Phi_n \to \Phi$ a.e. as $n\to \infty$. Define
$$\bar{v}(x) := E\left[\Phi(X_t^x) \int_0^t a(s) \left[\sigma^{-1}(s,X_s^x,\pi_s^x)Y_s u(t)\right]^\ast\delta W_s\right]^\ast.$$
To make reading clearer introduce the notation $C:= E[|\Phi(X_t^x)|^2]^{1/2}$ and the matrix -valued process $\xi_s := a(s) \sigma^{-1}(s,X_s^x,\pi_s^x)Y_s u(t)$, $0\leq s\leq t$. Then the objects $v(x):=E[\Phi(X_t^x)]$ and $\bar{v}(x)$ are well-defined since $\Phi(X_t^x)\in L^2(\Omega)$ and using Cauchy-Schwarz' inequality we have
$$|\bar{v}(x)| \leq C E\left[\int_0^t \tr[\xi_s \xi_s^\ast] ds + \int_0^t \int_0^t D_s \xi_r D_r \xi_s dr ds\right]^{1/2}$$
where we used It\^{o}'s isometry property for Skorokhod integrals, see Theorem \ref{duality} or e.g. \cite[Theorem 6.17.]{DOP08}. Observe that the first term is bounded since $a$ and $\sigma^{-1}$ are uniformly bounded and $u$ has integrable trajectories. The second term is bounded since $\xi_s$ is Malliavin differentiable for every $s\in [0,t]$ because $Y_su(t)$ is Malliavin differentiable for every $s\in [0,T]$ in virtue of Proposition \ref{detYt} in connection with Proposition \ref{Malldiff} as for $u(t)$, due to Corollary \ref{SkorokhodII}.

Now, we approximate $v$ by $v_n(x) := E[\Phi_n(X_t^x)]$. It is clear that $v_n \to v$ a.e. and now we can use the Bismut-Elworthy-Li formula on $\frac{\partial}{\partial x}v_n$ in order to estimate $|\frac{\partial}{\partial x}v_n(x) - \bar{v}(x)|$. Indeed, again by Cauchy-Schwarz inequality and It\^{o}'s isometry for the Skorokhod integral we have
\begin{align*}
\left|\frac{\partial}{\partial x}v_n(x) - \bar{v}(x)\right| &\leq E[|\Phi_n(X_t^x) - \Phi(X_t^x)|^2]^{1/2} E[|w(t)|^2]^{1/2}
\end{align*}
where $w(t):=\int_0^t \xi_s^\ast \delta W_s$ denotes the Malliavin weight. Now since $\Phi_n$ and $\Phi$ are continuous and bounded we have for every compact subset $K\subset \R^d$
\begin{align*}
\lim_{n\to \infty}\sup_{x\in K} \left|\frac{\partial}{\partial x}v_n(x) - \bar{v}(x)\right| =0.
\end{align*}
Hence, $v$ is continuously differentiable with $\frac{\partial}{\partial x} v = \bar{v}$.

{\bf Step 3:} Let us denote
$$\mathcal{G}:= \{\Phi: \R^d\rightarrow \R_+ \, \mbox{ continuous and bounded} \}.$$

It is clear that $\mathcal{G}$ is a multiplicative class, i.e. $\psi_1,\psi_2\in \mathcal{G}$ then $\psi_1\psi_2\in \mathcal{G}$. Further, let $\mathcal{H}$ be the class of functions $\Phi: \R^d\rightarrow \R_+$ for which \eqref{BEL} holds. From Step 2 we have $\mathcal{G} \subset \mathcal{H}$. Then $\mathcal{H}$ is a monotone vector space on $\R^d$, see e.g. \cite[p.23]{Prot05} for definitions. Indeed, from dominated convergence we have monotonicity. In fact, if $\{\Phi_n\}_{n\geq 0}\subset \mathcal{H}$ such that $0\leq \Phi_1\leq \cdots \leq \Phi_n\leq \cdots$ with $\lim_n \Phi_n=\Phi$ and $\Phi$ is bounded then $\Phi \in \mathcal{H}$. Furthermore, denote by $\sigma(\mathcal{G}):= \{f^{-1}(B), \, B\in \mathcal{B}(\R_+), \, f\in \mathcal{G}\}$ where $\mathcal{B}(\R_+)$ denotes the Borel $\sigma$-algebra in $\R_+$. Then we are able to apply the monotone class theorem, see e.g. \cite[Theorem 8]{Prot05} and conclude that $\mathcal{H}$ contains all bounded and $\sigma(\mathcal{G})$-measurable functions $\Phi:\R^d \rightarrow \R_+$. Nevertheless, $\sigma(\mathcal{G})$ coincides with the Borel $\sigma$-algebra of $\R^d$ since $\mathcal{G}$ contains all continuous bounded functions. So we conclude that $\mathcal{H}$ contains all bounded Borel measurable functions on $\R^d$.

{\bf Step 4:} The last step is then to approximate any $\mathcal{B}(\R^d)$-measurable function $\Phi:\R^d \rightarrow \R_+$ such that $\Phi (X_t^x)\in L^2(\Omega)$ by a sequence $\{\Phi_n\}_{n\geq 0}$ of bounded $\mathcal{B}(\R^d)$-measurable functions. For example,
$$\Phi_n(x) = \Phi(x) 1_{\{\Phi(x)\leq n\}}, \quad x\in \R^d, \quad n\geq 0.$$

Then $\Phi_n\in \mathcal{H}$ for each $n\geq 0$. Define $\widetilde{v}(x):=E[\Phi(X_t^x)w(t)]$. Then by Cauchy-Schwarz' inequality and It\^{o}'s isometry we know
$$\sup_{x\in K}\left|\frac{\partial}{\partial x} v_n(x) - \widetilde{v}(x)\right|\leq C \sup_{x\in K} E\left[|\Phi_n(X_t^x) - \Phi(X_t^x)|\right]^{1/2},$$
for any compactum $K\subset \R^d$ and some finite constant $C>0$. Finally, observe that clearly one has
$$\sup_{x\in K}E\left[|\Phi_n(X_t^x) - \Phi(X_t^x)|\right]^{1/2} \xrightarrow{n\to \infty} 0$$
thus proving the result.
\end{proof}

\section{Applications}\label{SectionAppl}

In this section we wish to give a rather simple but illustrative example of how the dependence on the expectation of the solution may give rise to more complicated terms when deriving the Bismut-Elworthy-Li formula. In one of the examples we adopt the context of finance where the formula has a broad use for the computation of the so-called Delta sensitivities which, in short, is the sensitivity of prices of contracts with respect to the initial value of the price of the stock taken into consideration. We will consider the price of an option written on a stock whose dynamics depend on the expectation of the price process. Then we provide two numerical examples in order to demonstrate that the Bismut-Elworthy-Li formula, or the so-called Malliavin method for computing the Delta is numerically more efficient than the usual finite difference method even when the function $\Phi$ is discontinuous. 

\begin{ex}[Black-Scholes model with continuous dividend payments]
Let $S=\{S_t^x, t\in [0,T]\}$ represent the price dynamics of some asset with initial price $x>0$ governed by the following SDE
\begin{align}\label{II_BSmodel}
\frac{dS_t^x}{S_t^x} = (\mu- q \rho_t^x) dt +  \sigma dW_t, \quad \rho_t^x:=E[S_t^x], \quad t\in [0,T], \quad S_0^x=x>0
\end{align}
where $\mu, q, \sigma\in \R$ and $\sigma>0$. Let $S_t^0 = e^{rt}$, $t\in [0,T]$, $r\in \R$ with $r>0$ be the risk-less asset and $\Phi:\R\rightarrow [0,\infty)$ a pay-off function.

Then the price of a European option at current time with maturity $T>0$ (under the risk-neutral valuation approach) is given by
$$p_T(x) = e^{-rT} E_{\widetilde{P}}\left[\Phi(S_T^x)\right]$$
where $\widetilde{P}$ is the risk-neutral measure, i.e.
$$\frac{d\widetilde{P}}{dP}\bigg|_{\mathcal{F}_t} = M_t^x := e^{-\int_0^t \theta_u^x dW_u - \frac{1}{2}\int_0^t (\theta_u^x)^2du}, \quad t\in [0,T]$$
where
$$\theta_t^x := \frac{\mu - r -q\rho_t^x}{\sigma}, \quad t\in [0,T]$$
is the market price of risk process.

It follows that $\rho_t^x = \frac{x\mu e^{\mu t}}{qxe^{\mu t} + \mu - qx}$ obtained as the solution of a Riccati equation. Also, we have $\frac{\partial}{\partial x}\rho_T^x = \frac{e^{-\mu t}}{x^2}(\rho_t^x)^2$ and $\frac{\partial}{\partial x}\theta_t^x = -\frac{q}{\sigma}\frac{e^{-\mu t}}{x^2}(\rho_t^x)^2$. Then the $\Delta$-sensitivity of an option $\Phi$ written on $S_T^x$ is given by
$$\Delta=e^{-rT} E\left[\Phi(S_T^x) \left(\frac{\partial}{\partial x}M_T^x  +  \frac{1}{x^2\sigma}e^{-\mu T}\rho_T^x\int_0^T a(s)M_T^x \delta W_s\right)\right].$$

Let us find a simpler expression for the stochastic integral. Using the integration by parts formula for the Skorokhod integral, see Theorem \ref{IBP}, we find that
$$\int_0^T M_T^x \delta W_s = \left( W(T) - \int_0^T \theta_s^x ds \right)M_T^x$$
and hence, taking $a\equiv \frac{1}{T}$ we find that under the risk-neutral measure $\widetilde{P}$, the $\Delta$-sensitivity is given by
$$\Delta=e^{-rT} E_{\widetilde{P}} \left[\Phi(S_T^x)Z_T\right]$$
with Malliavin weight
$$Z_T:=\frac{q}{x^2 \sigma}\left(\int_0^T e^{-\mu s}(\rho_s^x)^2 dW_s + \int_0^T \theta_s^x e^{-\mu s} (\rho_s^x)^2 ds + \frac{e^{-\mu T}}{q} \rho_T^x \frac{1}{T}\left( W(T) - \int_0^T \theta_s^x ds\right)\right).$$

Finally, observe that if we ignore the dependence on $E[S_t^x]$, e.g. taking $q = 0$ then we obtain
$$\Delta = e^{-rT} E_{\widetilde{P}} \left[\Phi(S_T^x) \frac{W(T) - \int_0^T \theta_s^x ds}{Tx\sigma}\right]=  e^{-rT} E_{P} \left[\Phi(S_T^x) \frac{\hat{W}(T)}{Tx\sigma}\right],$$ where $\hat{W}$ is a standard Brownian motion under $P$ and hence the $\Delta$ coincides with the classical one.
\end{ex}

\begin{ex}
Consider now the following SDE,
\begin{align}\label{II_GeomSDE}
\frac{dX_t^x}{X_t^x} = f(\rho_t^x) dt +  \sigma dW_t, \quad \rho_t^x:=E[X_t^x], \quad t\in [0,T], \quad X_0^x=x>0,
\end{align}
for some suitable continuously differentiable function $f:\R \rightarrow \R$. Let us compute the derivative of $v(x):=E[\Phi(X_T^x)]$ for an irregular function $\Phi$.

Using Theorem \ref{BELthm} with $a(s) \equiv 1/T$ we have
$$v'(x) = E\left[ \Phi (X_T^x) \frac{1}{T}\int_0^T \frac{1}{\sigma X_s} \frac{X_s}{x} u(T) \delta W_s\right] = \frac{1}{\sigma x T} u(T) E\left[\Phi(X_T^x) W_T\right]$$
where
$$u(T) = 1+x \int_0^T f'(\rho_s^x) \frac{\partial}{\partial x} \rho_s^x ds.$$

This special "geometric-type" case shows that whenever $u$ is deterministic then the delta $v'$ is an rescaled version of the classical delta and they coincide when $f$ is constant, indeed.
\end{ex}

Let us then see what happens when we make the volatility coefficient depend on the expectation, which leads to a stochastic $u$.

\begin{ex}
Consider now
\begin{align}\label{II_GeomSDE2}
\frac{dX_t^x}{X_t^x} = \mu dt +  \sigma \rho_t^x dW_t, \quad \rho_t^x:=E[X_t^x], \quad t\in [0,T], \quad X_0^x=x>0.
\end{align}

In this case according to Theorem \ref{MallSob} in connection with Theorem \ref{BELthm} the Malliavin weight, denoted by $w_T$,  becomes
$$w_T^x:= \frac{1}{\sigma x T} \int_0^T \frac{1}{\rho_s^x} \left(1-\sigma^2 \int_0^T (\rho_s^x)^2 ds + \sigma \int_0^T \rho_s^x dW_s\right)\delta W_s.$$

Fortunately we can rewrite the above expression in terms of an It\^{o} integral using integration by parts. Namely,
$$w_T^x = \left(1-\sigma^2 \int_0^T (\rho_s^x)^2 ds + \sigma \int_0^T \rho_s^x dW_s\right) \int_0^T \frac{1}{\rho_s^x}dW_s - \sigma T.$$

Denote the random variables $F:=\int_0^T \rho_s^x dW_s$ and $G:= \int_0^T \frac{1}{\rho_s^x}dW_s$. Then the vector $(F,G)$ is normally distributed with zero mean and covariance matrix
$$\Sigma := \begin{pmatrix} \int_0^T \rho_s^2 ds & T \\ T & \int_0^T \frac{1}{\rho_s^2} ds\end{pmatrix}.$$

Altogether we obtain
$$v'(x) = \frac{1}{\sigma x T} E\left[\Phi (X_T^x)\left( 1-\sigma^2 \int_0^T \rho_s^2 ds + \sigma F\right) G  \right] - \frac{1}{x} E[\Phi(X_T^x)].$$

Let, for instance, $\Phi(x) = (x-K)1_{\{x\geq K\}}$ for some fixed $K>0$, also known as a \emph{European call option} in the context of finance. Then we use a Monte Carlo method to compute the above expression and compare it to the following finite difference method scheme hereunder
$$v'(x) \approx \frac{E[\Phi(X_T^{x+h})]-E[\Phi(X_T^{x})]}{h}, \quad h\approx 0.$$

\begin{figure}[H]
\centering
\includegraphics[scale=0.6]{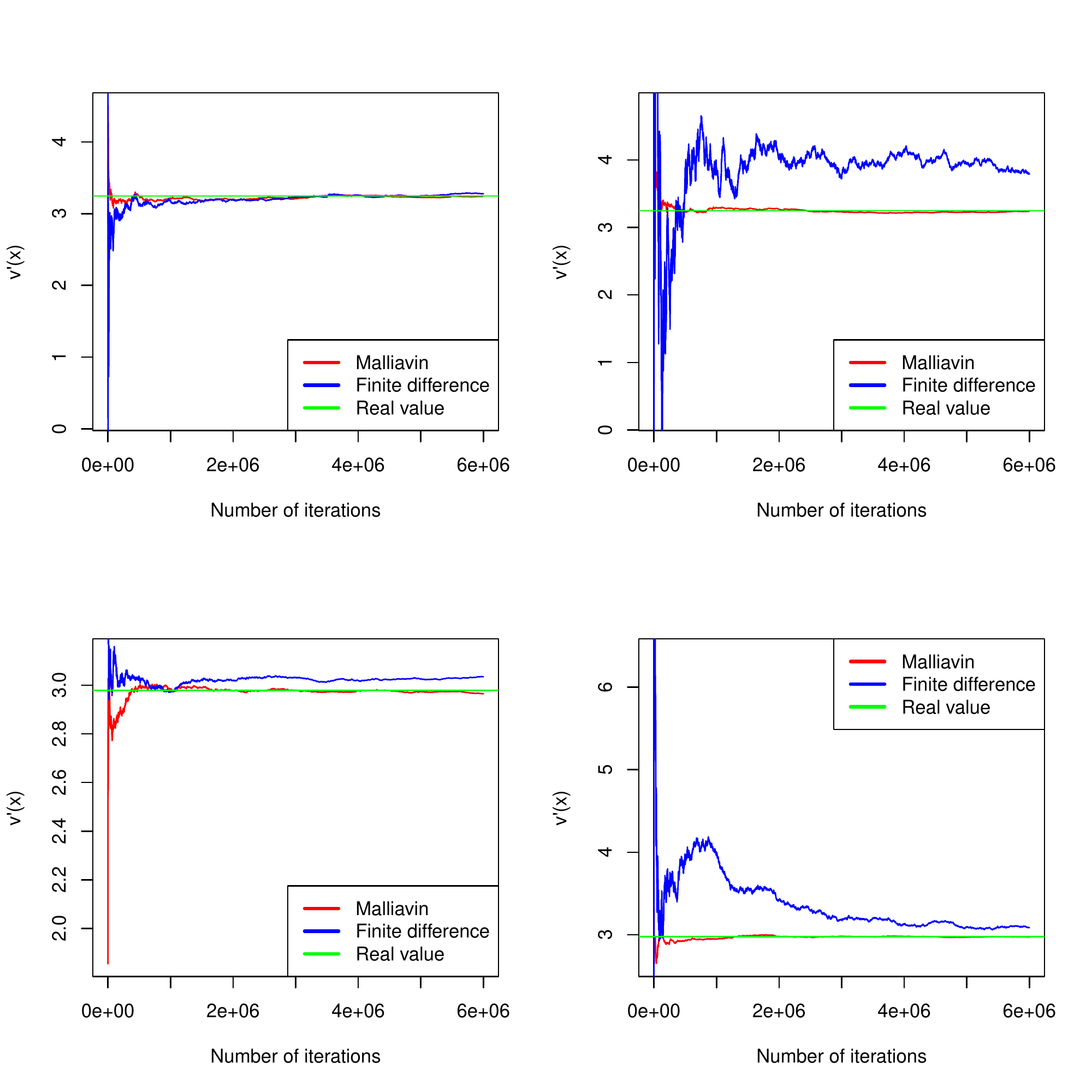}
\caption{Approximation of a Call Option: On top parameters set to $\sigma=0.8$, $\mu=1$, $x=1$, $K=2$ and $T=1$ with $h=0.1$ to the right and $h=0.01$ to the left. On bottom parameters set to $\sigma=1.2$, $\mu=1$, $x=0.5$, $K=0.7$ and $T=1$ with $h=0.1$ to the right and $h=0.01$ to the left}
\end{figure}

In the upper left figure, $h=0.1$ for the finite difference method and the two methods are seemingly giving similar accurate results, although the Malliavin method is more efficient in number of iterations. If one wishes to decrease $h$ in order to gain precision we can see how the finite difference method becomes unstable (upper right and lower right figures). In conclusion, the integration by parts formula seems to be a much more efficient tool for the computation of sensitivities for mean-field SDEs, at least, in this setting.

Let us now try a more irregular function $\Phi$, namely $\Phi(x) = 1_{\{x\geq K\}}$ which has a discontinuity at $x=K$, also known as a \emph{European digital option} in the context of finance. Denote
$$d(x) := \frac{\log \frac{K}{x} - \mu T + \frac{1}{2}\sigma^2 \int_0^T (\rho_s^x)^2 ds}{\sigma}.$$

Then
$$v'(x) = \frac{1}{\sigma x T} E\left[ 1_{\{F \geq d(x)\}} \left(\left(1-\frac{\sigma^2 x^2}{2\mu}(e^{2\mu T} - 1) + \sigma F \right)G - \sigma T \right) \right].$$

We compare again the Malliavin method with a finite difference scheme.

\begin{figure}[H]
\centering
\includegraphics[scale=0.6]{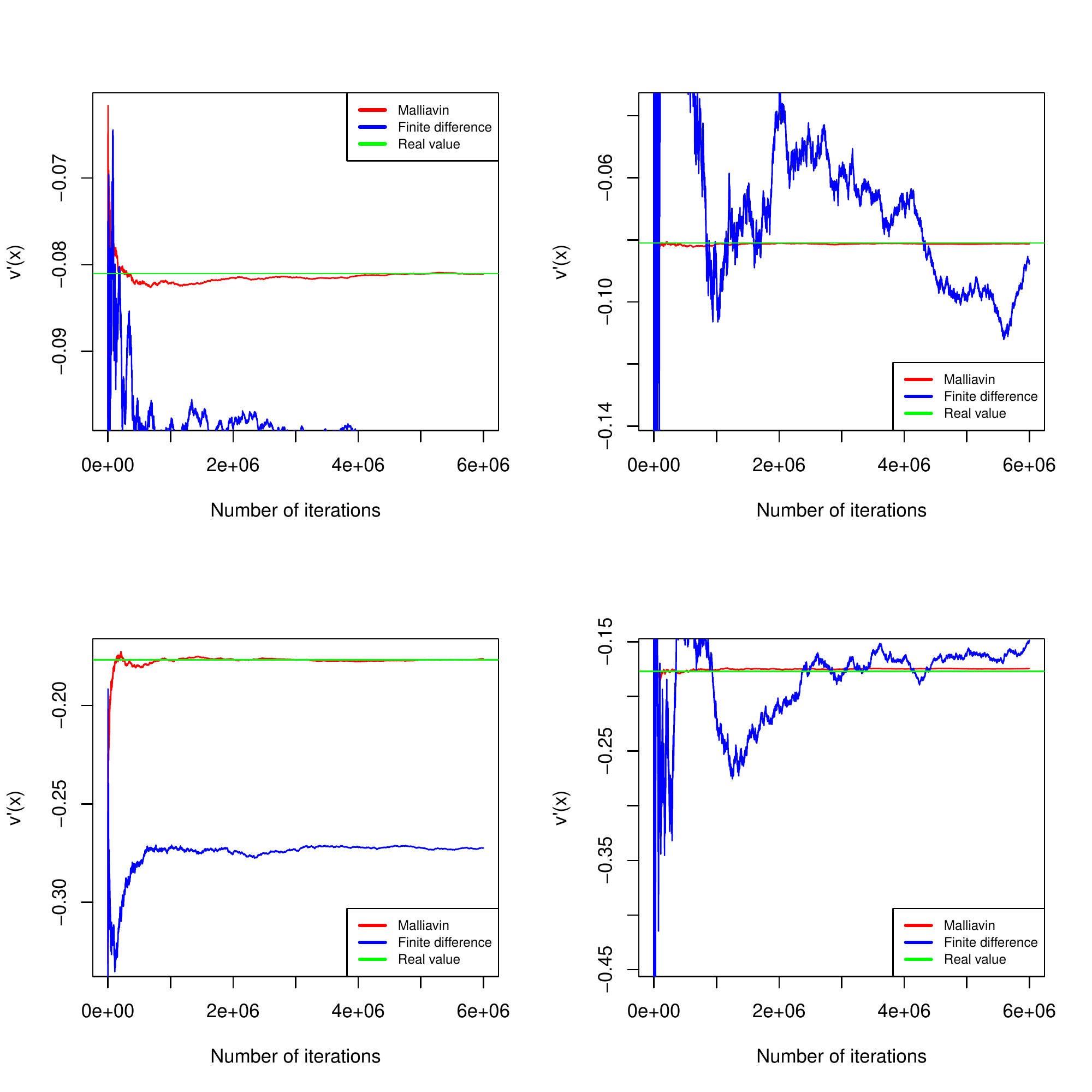}
\caption{Approximation of a Digital Option: On top parameters set to $\sigma=0.8$, $\mu=1$, $x=1$, $K=2$ and $T=1$ with $h=0.1$ to the right and $h=0.01$ to the left. On bottom parameters set to $\sigma=1.2$, $\mu=1$, $x=0.5$, $K=0.7$ and $T=1$ with $h=0.1$ to the right and $h=0.01$ to the left}
\end{figure}
\end{ex}

The conclusions here are clear. The regularity of the function $\Phi$ plays an important role. We see that the bias in the finite difference method seems high and it becomes unstable when decreasing the values of $h$. On the contrary, the Bismut-Elworthy-Li formula gives a better approximation of the sensitivity, even when the function $\Phi$ is discontinuous.

\end{document}